\documentclass[pdflatex,sn-mathphys-num]{sn-jnl}

\usepackage{graphicx}%
\usepackage{multirow}%
\usepackage{amsmath,amssymb,amsfonts}%
\usepackage{amsthm}%
\usepackage{mathrsfs}%
\usepackage[title]{appendix}%
\usepackage{xcolor}%
\usepackage{textcomp}%
\usepackage{manyfoot}%
\usepackage{booktabs}%
\usepackage{algorithm}%
\usepackage{algorithmicx}%
\usepackage{algpseudocode}%
\usepackage{listings}%
\usepackage{tikz}
\usetikzlibrary{positioning, arrows.meta, shapes.geometric, calc, fit, backgrounds}

\newcommand{\R}{\mathbb{R}}
\newcommand{\Rn}{\mathbb{R}^n}
\newcommand{\Uad}{U_{\mathrm{ad}}}
\newcommand{\Xad}{X_{\mathrm{ad}}}
\newcommand{\Sop}{\mathcal{S}}
\newcommand{\Aop}{\mathcal{A}}
\newcommand{\Bop}{\mathcal{B}}
\newcommand{\Eop}{\mathcal{E}}
\newcommand{\Gop}{\mathcal{G}}
\newcommand{\Lc}{\mathcal{L}}
\newcommand{\Wsp}{W}
\newcommand{\inner}[2]{\langle #1,#2\rangle}
\newcommand{\norm}[1]{\lVert #1\rVert}

\DeclareMathOperator{\proj}{proj}
\DeclareMathOperator{\aff}{aff}

\theoremstyle{thmstyleone}
\newtheorem{theorem}{Theorem}
\newtheorem{lemma}[theorem]{Lemma}
\newtheorem{corollary}[theorem]{Corollary}
\newtheorem{proposition}[theorem]{Proposition}%
\newtheorem{assumption}{Assumption}

\theoremstyle{thmstyletwo}%
\newtheorem{example}{Example}%
\newtheorem{remark}{Remark}%

\theoremstyle{thmstylethree}%
\newtheorem{definition}{Definition}%

\begin{document}

\title[Inverse Optimal Control with Convex Features and Box Constraints]
{Inverse Optimal Control with Convex Features and Box Constraints}

\author[1]{\fnm{Jiguang} \sur{Yu}}\email{jyu678@bu.edu}
\equalcont{These authors contributed equally to this work.}

\author*[2]{\fnm{Louis Shuo} \sur{Wang}}\email{wang.s41@northeastern.edu}
\equalcont{These authors contributed equally to this work.}

\affil[1]{\orgdiv{College of Engineering}, \orgname{Boston University}, \orgaddress{\city{Boston}, \postcode{02215}, \state{MA}, \country{USA}}}

\affil*[2]{\orgdiv{Department of Mathematics}, \orgname{Northeastern University}, \orgaddress{\city{Boston}, \postcode{02115}, \state{MA}, \country{USA}}}

\abstract{We study an inverse optimal-control problem for identifying nonnegative weights in a
convex-feature lower-level objective. The resulting model is an optimistic bilevel
optimal-control problem whose lower level is a strongly convex, linearly constrained
control problem with box constraints in \(L^2\). For the reduced response map
\(x\mapsto (y_x,u_x)\), we prove global single-valuedness and an explicit Lipschitz
estimate, without assuming differentiability of active sets. Exploiting the affine
dependence of the lower-level functional on \(x\), we show that the value function is
concave and Lipschitz and derive an asymptotically exact optimal-value relaxation.
We further establish local identifiability and \(O(\delta)\) noise stability by a
gradient-free argument based on second-order growth. Finally, a KKT reformulation
yields a function-space MPCC; Scholtes relaxation gives C-stationarity under uniform
multiplier boundedness. A decoupled scalar example provides an explicit
strict-complementarity certificate via finiteness of threshold contacts.}



\maketitle

\section{Introduction and Problem Formulation}
\label{sec:intro}

\subsection{Motivation}
\label{subsec:motivation}

In many applications one observes the optimal behaviour of a controlled system but
does not know the objective that the system is optimizing. Inverse optimal
control -- equivalently, parameter learning for variational and optimal-control
models -- seeks to recover that objective, or the weights of its competing terms,
from measured state and control trajectories \cite{ren2026inverse,deng2001stabilization,do2015global,wang2025analysis}. The problem is now classical in
robotics and biomechanics, where the weights of a parametrized locomotion cost are
estimated from motion-capture data \cite{mombaur2010human,hatz2012estimating,albrecht2012bilevel,wang2025analysis1}, and it
is central to data-driven variational modelling in imaging, where regularization
operators and weights are learned from training pairs
\cite{kunisch2013bilevel,de2013image,de2017bilevel,liang2025global}.
In every case the learned model is itself an optimization problem, so the
identification task is genuinely hierarchical: an upper-level data-misfit is
minimized subject to the constraint that the lower-level state and control solve a
parametrized optimal control problem. This is precisely an optimistic bilevel
program \cite{dempe2002foundations,zemkohoo2020bilevel,wang2026algebraic,bard2013practical,dempe2020bilevel,sinha2017review}.

The present paper treats a structured instance of this problem that arises when the
lower-level objective is a nonnegative combination of parameterized convex
features $j_1,\dots,j_n$ -- which we also refer to, following the application
interpretation, as competition features. Let $x\in\Xad\subset\R^n_+$ denote the
unknown vector of feature weights. For each $x$, the lower-level problem is a strongly
convex constrained optimal control problem of the form
\[
  \min_{(y,u)}
  \left\{
    x\cdot j(y)+\frac{\sigma}{2}\norm{u}_U^2
  \right\}
  \quad\text{subject to the linear state equation and } u\in\Uad.
\]
The inverse problem asks for a parameter $x$ whose optimal lower-level response
matches observed data; this yields an optimistic bilevel optimal control problem in
which the lower-level solution map is generally nonsmooth because of box constraints
on the control.

The mathematical questions we address are: existence and stability of the
lower-level response; structure of the lower-level value function; exactness of the
value-function relaxation; identifiability and noise stability of the inverse
solution; and stationarity of local minimizers after a complementarity
reformulation.

\subsection{Related work and positioning}
\label{subsec:related}

\paragraph{Bilevel optimization and the value-function reformulation.}
Optimality conditions for bilevel programs are notoriously delicate because standard
constraint qualifications fail at every feasible point of the natural
single-level reformulations. Two reformulations dominate the literature. The
optimal-value reformulation replaces the lower-level problem by the scalar
constraint $f(x,y,u)\le\varphi(x)$, where $\varphi$ is the lower-level value
function; this approach was developed by Ye and Zhu, who identified calmness as the
appropriate qualification \cite{ye1995optimality,ye2010new,wang2026damage}. The KKT reformulation
replaces the lower level by its first-order system and produces a mathematical
program with complementarity (or equilibrium) constraints
\cite{luo1996mathematical,outrata2013nonsmooth,dempe2002foundations,wang2026breakdown}. We use the value-function
form for existence and exact-relaxation arguments, where its concavity structure is
an asset, and the KKT form for stationarity, where pointwise complementarity is
natural; the comparison after \eqref{eq:OVR} explains why we do not derive
stationarity from the value-function form.

\paragraph{MPCCs and the stationarity hierarchy.}
Programs with complementarity constraints possess a hierarchy of stationarity
notions -- weak, Clarke (C-), Mordukhovich (M-), and strong (S-) stationarity --
introduced and analysed by Scheel and Scholtes \cite{scheel2000mathematical} and
Scholtes \cite{scholtes2001convergence}, with a systematic comparison of relaxation methods by
Hoheisel, Kanzow and Schwartz \cite{hoheisel2013theoretical}. A robust feature of Scholtes-type
relaxation is that, without additional structure, multiplier limits satisfy only the
Clarke sign condition on the biactive set; M- or S-stationarity require stronger
qualifications. Our Theorem~\ref{thm:cstationarity} reflects this: we obtain
C-stationarity under a uniform multiplier bound, and we are explicit that strong
stationarity is not available in general.

\paragraph{Complementarity and bilevel problems in function space.}
The transfer of the finite-dimensional MPCC theory to function space is subtle
because pointwise-defined complementarity sets are typically non-polyhedric and the
relevant normal cones are not weakly closed in $L^2$. The variational geometry of
such sets, and the resulting weak/M-stationarity systems, were clarified by Mehlitz
and Wachsmuth \cite{mehlitz2018limiting,mehlitz2016weak} and Harder and
Wachsmuth \cite{harder2022m}; elliptic MPECs with control constraints were
treated by Hinterm\"uller, Mordukhovich and Surowiec \cite{hintermuller2014several}; control
complementarity constraints and their penalization were studied by Deng, Mehlitz and
Pr\"ufert \cite{deng2021coupled}. Our function-space C-stationarity result
is stated in the pointwise, almost-everywhere form natural for box control
constraints and uses exactly this circle of ideas to pass to the limit.

\paragraph{Inverse problems and regularization.}
The identification viewpoint connects the bilevel problem to the theory of inverse
problems: existence of inverse solutions, identifiability, and stability under noise
in the spirit of regularization theory \cite{engl1996regularization,ito2014inverse,holler2018bilevel,yu2026rigorous}. The
PDE-constrained and optimal-control background is standard
\cite{troltzsch2010optimal,hinze2008optimization,alexanderian2026computational,yu2026beyond}.

\paragraph{Contribution relative to the literature.}
Existing inverse-optimal-control work is largely either finite-dimensional /
discrete-time \cite{johnson2013inverse,ab2020inverse,chan2025inverse,hatz2012estimating,mombaur2010human,zhang2022statistically,yu2026microscopic} or focuses on imaging variational models
with elliptic lower levels \cite{kunisch2013bilevel,de2013image,cai2026optimal}. The
function-space MPCC analyses \cite{mehlitz2016weak,harder2022m,harder2018limiting,harder2020bilevel,wang2026elliptic} are
general but do not specialize the constants or the strict-complementarity structure
to a constrained LQ setting. The specific contribution here is not the introduction
of a new MPCC stationarity concept, but the combination, for a parameterized convex
constrained optimal-control lower level, of: (i) a global, explicit Lipschitz
estimate for the parameter-to-response map; (ii) a concave value-function relaxation
that is exact in the limit; (iii) inverse identifiability and $O(\delta)$ noise
stability established without differentiability of the response; and (iv) a
verifiable scalar strict-complementarity certificate for the MPCC qualification
underlying C-stationarity.

\subsection{Contributions}
\label{subsec:contributions}
\begin{enumerate}[label=\textup{(C\arabic*)}]
\item \textbf{Lower-level well-posedness and stability.}
For bounded box controls, convex $C^2$ features, and $\sigma>0$, we
prove existence and uniqueness of the lower-level response and Lipschitz continuity
of the response map $x\mapsto(y_x,u_x)$, with an explicit constant. We do not claim
$C^1$-smoothness in the constrained case; smoothness is asserted only under
additional active-set stability assumptions or in unconstrained/quadratic special
cases.

\item \textbf{Value-function structure.}
We prove that the lower-level value function $\varphi$ is concave and Lipschitz on
$\Xad$, and that $j(y_x)$ is a concave supergradient. This justifies the
optimal-value reformulation and its $\varepsilon$-relaxation.

\item \textbf{Exact relaxation and inverse stability.}
We prove that global minimizers of the relaxed value-function problems converge,
along subsequences, to global minimizers of the original bilevel problem as
$\varepsilon\downarrow0$. We add local identifiability under an injective sensitivity
condition, stated for the data metric actually present in the objective, and
$O(\delta)$ stability under noisy observations under a second-order growth condition,
proved by a gradient-free argument.

\item \textbf{MPCC stationarity.}
We replace the lower-level problem by its KKT system -- which is equivalent to the
bilevel problem because the lower level is convex -- express the box normal cone as
pointwise complementarity, and apply a Scholtes-type relaxation. Under a uniform
boundedness condition on the relaxed multipliers (which would follow from a
function-space MPCC--MFCQ), the limits of relaxed local minimizers are C-stationary.
The framework follows \cite{scheel2000mathematical,scholtes2001convergence,hoheisel2013theoretical,liu2025bidirectional} and its
function-space counterparts \cite{mehlitz2018limiting,hintermuller2014several,liang2026separation}.

\item \textbf{Quadratic and scalar specializations.}
We include a self-contained quadratic specialization with explicit Lipschitz
constants, define the two-agent competition model used throughout, and prove a
decoupled scalar strict-complementarity theorem: when the coupling weight vanishes,
each scalar agent has finitely many threshold contacts and strict complementarity
holds almost everywhere (for the $L^2$ normal cone), without any nondegeneracy
condition. Active arcs may have positive measure; it is only the threshold-contact
times, and hence the biactive times, that are finite (and null).
\end{enumerate}

The paper is organized as follows. Section~\ref{sec:wellposed} establishes
lower-level well-posedness, Lipschitz stability, the quadratic and two-agent
specializations, identifiability, and noise stability. Section~\ref{sec:value}
studies the value function, the optimal-value reformulation, and exactness of the
value-function relaxation. Section~\ref{sec:mpcc} derives the KKT reformulation, the
MPCC relaxation, C-stationarity, and the scalar strict-complementarity result.
Technical MPCC definitions and the detailed scalar proof are in the appendix.

\subsection{Functional setting}
\label{subsec:functional-setting}
Let $T>0$ be fixed. We use the Hilbert spaces
\[
  Y:=H^1(0,T;\R^d),\qquad U:=L^2(0,T;\R^m).
\]
Controls are subject to box constraints
\[
  \Uad:=\{u\in U:\ u_a(t)\le u(t)\le u_b(t)\ \text{for a.e. }t\in(0,T)\},
\]
where the inequalities are componentwise and
\[
  u_a,u_b\in L^\infty(0,T;\R^m),\qquad u_a(t)\le u_b(t)\ \text{a.e.},
\]
and we assume $\Uad\neq\varnothing$. Because $\Uad$ is closed, convex, and bounded in
the reflexive space $U=L^2(0,T;\R^m)$, it is weakly sequentially compact
(Lemma~\ref{lem:box-basic}); this compactness drives the direct method in the
lower-level existence proof.

The state equation is linear,
\[
  \dot y(t)=\Aop y(t)+\Bop u(t)\ \text{for a.e. }t\in(0,T),\qquad y(0)=y_0,
\]
with $\Aop\in\R^{d\times d}$, $\Bop\in\R^{d\times m}$, $y_0\in\R^d$. For every
$u\in U$ this has a unique solution $y_u\in Y$, so the control-to-state map
$\Sop:U\to Y$, $u\mapsto y_u$, is affine and continuous,
\[
  \Sop u=y^0+Su,\qquad S\in\Lc(U,Y),
\]
see Lemma~\ref{lem:control-to-state}. We use the operator notation
\[
  \Eop y-\Gop u=b_0,\qquad
  \Eop y:=(\dot y-\Aop y,\,y(0)),\quad
  \Gop u:=(\Bop u,0),\quad b_0:=(0,y_0).
\]

\subsection{Standing assumptions}
\label{subsec:assumptions}

\begin{assumption}[State--control setting]
\label{ass:A}
\hfill
\begin{enumerate}[label=\textup{(\roman*)}]
\item $\Xad\subset\R^n_+$ is nonempty, compact, and convex.
\item $\Uad\subset U$ is nonempty, closed, convex, and bounded.
\item The state equation is linear and $\Sop u=y^0+Su$ with $S\in\Lc(U,Y)$.
\end{enumerate}
\end{assumption}

\begin{assumption}[Competition features]
\label{ass:B}
For each $\ell=1,\dots,n$, $j_\ell:Y\to\R_+$ is convex and twice continuously
Fr\'echet differentiable. For the stability analysis we assume
\[
  M_j:=\sup_{y\in\Sop(\Uad)}\norm{j'(y)}_{\Lc(Y,\Rn)}<\infty.
\]
A sufficient, often easier-to-check condition is that $j'$ maps bounded subsets of
$Y$ to bounded subsets of $\Lc(Y,\Rn)$; the analysis below needs only the (weaker)
boundedness of $j'$ on $\Sop(\Uad)$ expressed by $M_j<\infty$.
\end{assumption}

\begin{remark}
In infinite dimensions, continuity of $j'$ alone does not imply boundedness on
bounded subsets. The bounded-derivative hypothesis is therefore not a technicality;
it is the condition needed to obtain a global Lipschitz estimate for the lower-level
response map on $\Xad$. By contrast, the codomain $\R_+$ (nonnegativity of $j_\ell$)
is included only for the application interpretation and plays no analytical role: the
results below use the convexity of $j_\ell$ together with $x\in\R^n_+$ to obtain
convexity of $x\cdot j$, and the bounded-derivative hypothesis for stability; none of
the proofs invoke $j_\ell\ge0$.
\end{remark}

\begin{assumption}[Strong control convexity]
\label{ass:C}
The control regularization parameter satisfies $\sigma>0$.
\end{assumption}

\begin{assumption}[Upper-level objective]
\label{ass:D}
The upper-level objective is
\[
  F(x,y,u)=\frac12\norm{y-y^{\mathrm{obs}}}_{L^2(0,T;\R^d)}^2
  +\frac{\beta}{2}\norm{u-u^{\mathrm{obs}}}_U^2+R(x),
\]
where $\beta\ge0$. For existence results, $R:\Xad\to\R$ is weakly lower semicontinuous; for
stationarity results, $R$ is $C^1$ on a neighborhood of $\Xad$.
\end{assumption}

The lower-level problem for fixed $x\in\Xad$ is
\begin{equation}
\tag{$P(x)$}\label{eq:Px}
  \min_{(y,u)}\left\{x\cdot j(y)+\frac{\sigma}{2}\norm{u}_U^2:\
  \Eop y-\Gop u=b_0,\ u\in\Uad\right\}.
\end{equation}
The inverse bilevel problem is
\begin{equation}
\tag{IOC}\label{eq:IOC}
  \min_{x,y,u}\left\{F(x,y,u):\ x\in\Xad,\ (y,u)\in\Psi(x)\right\},
\end{equation}
where $\Psi(x)$ denotes the solution set of \eqref{eq:Px}. Figures~\ref{fig:bilevel_framework} and~\ref{fig:bilevel-architecture} illustrate the schematic of the bilevel inverse control framework, including the upper level program and the lower level program.

\begin{figure}[htbp]
    \centering
    \includegraphics[width=\textwidth]{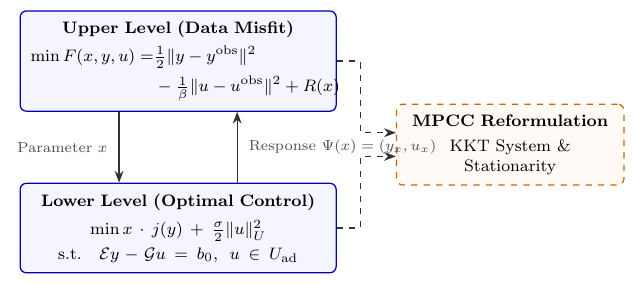}
    \caption{Schematic of the bilevel inverse optimal control framework. The upper-level problem minimizes the data misfit to optimize the parameter $x$, which is then passed to the lower-level optimal control problem to yield the response $\Psi(x)$. The right box represents the KKT and MPCC reformulations used for stationarity analysis.}
\label{fig:bilevel_framework}
\end{figure}

\begin{figure}
    \centering
    \includegraphics[width=\textwidth]{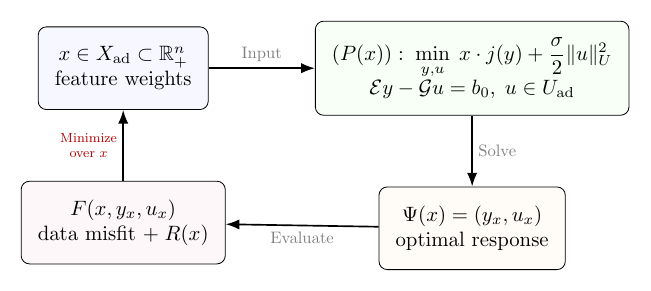}
    \caption{Optimistic bilevel structure of the inverse optimal-control problem.}
    \label{fig:bilevel-architecture}
\end{figure}

\section{Lower-Level Well-Posedness, Stability, and Identifiability}
\label{sec:wellposed}

\subsection{Reduced lower-level problem}
\label{subsec:reduced-lower}
Since $y=\Sop u$, problem \eqref{eq:Px} is equivalent to the control-only problem
\begin{equation}
\tag{$P_{\mathrm{red}}(x)$}\label{eq:Pred}
  \min_{u\in\Uad}\Phi_x(u),\qquad
  \Phi_x(u):=x\cdot j(\Sop u)+\frac{\sigma}{2}\norm{u}_U^2 .
\end{equation}

\begin{theorem}[Lower-level existence and uniqueness]
\label{thm:lower-existence-uniqueness}
Under Assumptions~\ref{ass:A}--\ref{ass:C}, for every $x\in\Xad$, \eqref{eq:Pred}
admits a unique minimizer $u_x\in\Uad$. Consequently $y_x:=\Sop u_x$ is uniquely
determined and the lower-level solution map $\Psi(x):=(y_x,u_x)$ is single-valued.
\end{theorem}

\begin{proof}
Fix $x\in\Xad$. Since $x\in\R^n_+$ and each $j_\ell$ is convex, $y\mapsto x\cdot j(y)$
is convex on $Y$; since $\Sop$ is affine, $u\mapsto x\cdot j(\Sop u)$ is convex on
$U$. Hence $\Phi_x$ is $\sigma$-strongly convex because of $\tfrac\sigma2\norm{u}_U^2$,
so it has at most one minimizer over $\Uad$.

Let $\{u_k\}\subset\Uad$ be a minimizing sequence. By Lemma~\ref{lem:box-basic},
$\Uad$ is weakly sequentially compact, so after a subsequence $u_k\rightharpoonup\bar
u\in\Uad$. Since $S\in\Lc(U,Y)$, $\Sop u_k=y^0+Su_k\rightharpoonup y^0+S\bar u=\Sop\bar
u$ weakly in $Y$. Each $j_\ell$ is convex and continuous, hence weakly lower
semicontinuous (Lemma~\ref{lem:wlsci}), and so is $\norm{\cdot}_U^2$. Therefore
$\Phi_x(\bar u)\le\liminf_k\Phi_x(u_k)$, so $\bar u$ is a minimizer. Uniqueness follows
from $\sigma$-strong convexity.
\end{proof}

\begin{theorem}[Lipschitz stability of the response map]
\label{thm:lipschitz-response}
Under Assumptions~\ref{ass:A}--\ref{ass:C} and the bounded-derivative condition in
Assumption~\ref{ass:B}, there is $C_\Psi>0$ such that for all $x_1,x_2\in\Xad$,
\[
  \norm{u_{x_1}-u_{x_2}}_U+\norm{y_{x_1}-y_{x_2}}_Y\le C_\Psi\norm{x_1-x_2}_{\Rn}.
\]
One may take $C_\Psi=\dfrac{M_j\norm{S}_{\Lc(U,Y)}}{\sigma}\bigl(1+\norm{S}_{\Lc(U,Y)}\bigr)$.
\end{theorem}

\begin{proof}
Let $u_i:=u_{x_i}$, $y_i:=\Sop u_i$, $i=1,2$. Throughout, $\inner{\cdot}{\cdot}_{U^*,U}$
and $\inner{\cdot}{\cdot}_{Y^*,Y}$ denote duality pairings; $j'(y)^*x\in Y^*$ since
$j'(y)\in\Lc(Y,\Rn)$. The variational inequality for $u_i$ is
\[
  \inner{S^*j'(y_i)^*x_i+\sigma u_i}{\,v-u_i}_{U^*,U}\ge0\qquad\forall v\in\Uad .
\]
Testing $i=1$ with $v=u_2$, $i=2$ with $v=u_1$, and adding,
\[
  \sigma\norm{u_1-u_2}_U^2\le\inner{j'(y_2)^*x_2-j'(y_1)^*x_1}{\,y_1-y_2}_{Y^*,Y} .
\]
Decompose the right-hand side as
\[
  \inner{j'(y_1)^*(x_2-x_1)}{y_1-y_2}_{Y^*,Y}
  +\inner{\bigl(j'(y_2)-j'(y_1)\bigr)^*x_2}{y_1-y_2}_{Y^*,Y} .
\]
The second term is nonpositive by monotonicity of $\nabla\bigl(y\mapsto x_2\cdot
j(y)\bigr)$. Hence $\sigma\norm{u_1-u_2}_U^2\le
M_j\norm{x_1-x_2}_{\Rn}\norm{y_1-y_2}_Y$. Since $y_1-y_2=S(u_1-u_2)$,
\[
  \norm{u_1-u_2}_U\le\frac{M_j\norm{S}}{\sigma}\norm{x_1-x_2}_{\Rn},
\]
and $\norm{y_1-y_2}_Y\le\norm{S}\norm{u_1-u_2}_U$ gives the state estimate. Summing the
two bounds yields the stated $C_\Psi$.
\end{proof}

\begin{remark}[No default $C^1$-smoothness]
The map $x\mapsto\Psi(x)$ is Lipschitz in general but need not be Fr\'echet
differentiable, because the active set of the box constraint may change with $x$.
Under active-set stability and a second-order sufficient condition, directional
differentiability can be obtained; in smooth inactive regimes, or in the
unconstrained quadratic case, $\Psi$ is $C^1$.
\end{remark}

\subsection{Quadratic specialization}
\label{subsec:quadratic}

\begin{proposition}[Quadratic competition features]
\label{prop:quadratic}
Assume $j_\ell(y)=\tfrac12\inner{Q_\ell y}{y}_Y+\inner{q_\ell}{y}_Y+c_\ell$ with
$Q_\ell=Q_\ell^*\succeq0$. Then $j_\ell$ is convex, $C^\infty$, and analytic, with
$j_\ell'(y)=Q_\ell y+q_\ell$ and $j_\ell''(y)=Q_\ell$. On $\Sop(\Uad)$,
\[
  \norm{j'(y)}_{\Lc(Y,\Rn)}\le C_Q(1+\norm{y}_Y),
\]
with $C_Q$ depending only on $\{Q_\ell,q_\ell\}$. Since $\Uad$ is bounded and $\Sop$
continuous, $j'$ is bounded on $\Sop(\Uad)$, and the Lipschitz estimate of
Theorem~\ref{thm:lipschitz-response} holds with
$C_\Psi=\tfrac{M_j\norm{S}}{\sigma}(1+\norm{S}).$
\end{proposition}

\begin{proof}
The derivative formulae are immediate; convexity follows from $Q_\ell\succeq0$. As
$\Sop(\Uad)$ is bounded in $Y$, the affine map $Q_\ell y+q_\ell$ is bounded there.
\end{proof}

\begin{example}
Let \(\mathcal H_\ell\) be a real Hilbert space, \(r_\ell \in \mathcal H_\ell\) be a given reference profile, and \(C_\ell \in \mathcal L(Y, \mathcal H_\ell)\) be a bounded linear observation operator. If
\[
  j_\ell(y)
  =
  \frac12\|C_\ell y-r_\ell\|_{\mathcal H_\ell}^2,
\]
then \(j_\ell\) is twice continuously Fréchet differentiable on \(Y\) with its first and second derivatives given by
\[
  j_\ell'(y)
  =
  C_\ell^*(C_\ell y-r_\ell),
  \qquad
  j_\ell''(y)
  =
  C_\ell^*C_\ell,
\]
where \(C_\ell^* \in \mathcal L(\mathcal H_\ell, Y)\) denotes the Hilbert-adjoint operator of \(C_\ell\).
Thus, Assumption~\ref{ass:B} holds automatically on \(\mathcal S(U_{\rm ad})\).
\end{example}

\subsection{The two-agent competition model}
\label{subsec:two-agent}
We make explicit the two-agent model referenced in the scalar analysis of
Section~\ref{sec:mpcc}. Let $d=m=2$, write $y=(y_1,y_2)$, $u=(u_1,u_2)$, and consider
decoupled agent dynamics
\[
  \dot y_i=a_i y_i+b_i u_i,\qquad y_i(0)=y_{0,i},\qquad i=1,2,
\]
i.e.\ $\Aop=\operatorname{diag}(a_1,a_2)$, $\Bop=\operatorname{diag}(b_1,b_2)$. The
competition features are a coupling term and an individual term,
\[
  j_1(y):=\frac12\int_0^T\bigl(y_1(t)-y_2(t)\bigr)^2\,dt,\qquad
  j_2(y):=\frac12\int_0^T\bigl(c_1y_1(t)^2+c_2y_2(t)^2\bigr)\,dt,
\]
with $c_1,c_2\ge0$, so $n=2$ and $x=(x_1,x_2)\in\R^2_+$. Both features are of the
quadratic type of Proposition~\ref{prop:quadratic} ($Q_1\succeq0$ couples the two
states, $Q_2\succeq0$ is diagonal), hence Assumption~\ref{ass:B} holds. The
lower-level objective is
\[
  x_1\,j_1(y)+x_2\,j_2(y)+\frac\sigma2\norm{u}_U^2 .
\]
When the coupling weight vanishes, $x_1=0$, the objective separates as
\[
  \sum_{i=1}^2\int_0^T\Bigl(\tfrac12 x_2 c_i\,y_i(t)^2+\tfrac\sigma2 u_i(t)^2\Bigr)dt,
\]
subject to the decoupled dynamics and box constraints $u_{a,i}\le u_i\le u_{b,i}$.
Thus \eqref{eq:Pred} splits into two independent scalar constrained LQ problems of
the form \eqref{eq:scalar-problem-app} with running weight $c:=\tfrac12 x_2c_i$
(matching the $c\,y^2$ normalization of \eqref{eq:scalar-problem-app}). This is the
structure exploited in Theorem~\ref{thm:scalar-sc} and
Corollary~\ref{cor:decoupled-two-agent-sc-app}.

\subsection{Local identifiability}
\label{subsec:identifiability}
Let $x^\dagger\in\Xad$ generate ideal observations $(y^\dagger,u^\dagger)=\Psi(x^\dagger)$
and define the reduced noiseless misfit
\[
  \mathcal J_0(x):=\frac12\norm{y_x-y^\dagger}_{L^2}^2+\frac\beta2\norm{u_x-u^\dagger}_U^2 .
\]
On $Y\times U$ introduce the (semi)norm associated with the data metric in the
objective,
\[
  \norm{(\eta,v)}_{\beta}^2:=\norm{\eta}_{L^2}^2+\beta\norm{v}_U^2 .
\]

\begin{proposition}[Local identifiability under a sensitivity condition]
\label{prop:identifiability}
Assume $x^\dagger$ lies in the relative interior of $\Xad$ and $\Psi$ is Fr\'echet
differentiable at $x^\dagger$ relative to $\aff\Xad$, with sensitivity operator
$D\Psi(x^\dagger):T_{\Xad}(x^\dagger)\to Y\times U$. Suppose the sensitivity is
injective in the data metric, i.e.
\begin{equation}
\label{eq:inj-beta}
  \norm{D\Psi(x^\dagger)h}_{\beta}=0\ \Longrightarrow\ h=0
  \qquad\text{for all }h\in T_{\Xad}(x^\dagger).
\end{equation}
Then $x^\dagger$ is locally identifiable: there is a neighborhood $V$ of $x^\dagger$ in
$\Xad$ such that, for $x\in V$, $\Psi(x)=\Psi(x^\dagger)$ implies $x=x^\dagger$.
Equivalently, with a basis $\{e_1,\dots,e_r\}$ of $T_{\Xad}(x^\dagger)$, the Gram matrix
$G_{ij}:=\inner{D\Psi(x^\dagger)e_i}{D\Psi(x^\dagger)e_j}_{\beta}$ is positive definite.
\end{proposition}

\begin{proof}
Since $\Psi$ is Fr\'echet differentiable at $x^\dagger$ relative to $\aff\Xad$ and
$\Psi(x^\dagger)=(y^\dagger,u^\dagger)$, the residuals $y_x-y^\dagger$ and $u_x-u^\dagger$
vanish to first order, so for $h\in T_{\Xad}(x^\dagger)$,
\[
  \mathcal J_0(x^\dagger+h)=\frac12\norm{D\Psi(x^\dagger)h}_{\beta}^2+o(\norm{h}^2).
\]
By \eqref{eq:inj-beta} the quadratic form $h\mapsto\norm{D\Psi(x^\dagger)h}_\beta^2$ is
positive on the finite-dimensional space $T_{\Xad}(x^\dagger)$, equivalently the Gram
matrix $G$ is positive definite, so there is $\alpha>0$ with
$\norm{D\Psi(x^\dagger)h}_\beta^2\ge\alpha\norm{h}^2$. Hence
$\mathcal J_0(x^\dagger+h)\ge\mathcal J_0(x^\dagger)+\tfrac\alpha2\norm{h}^2+o(\norm{h}^2)$,
which yields local quadratic growth and a strict local minimum at $x^\dagger$, i.e.\
local identifiability. The argument uses only Fr\'echet differentiability of $\Psi$,
not $C^2$-smoothness.
\end{proof}

\begin{remark}[Role of $\beta$]
\label{rem:beta-role}
The injectivity must be posed in the metric $\norm{\cdot}_\beta$ actually present in
the objective, not in the full norm of $Y\times U$. If $\beta>0$, then
$\norm{\cdot}_\beta^2\ge\min(1,\beta)\norm{\cdot}_{Y\times U}^2$, so
\eqref{eq:inj-beta} is equivalent to injectivity of the full operator
$D\Psi(x^\dagger)$. If $\beta=0$, the controls are not observed, the misfit sees only
the state, and \eqref{eq:inj-beta} reduces to injectivity of the
state-sensitivity component $D\Psi_y(x^\dagger)$ alone; injectivity of the full
$D\Psi(x^\dagger)$ is then not sufficient. This is why the hypothesis is stated
in the $\beta$-metric.
\end{remark}

\subsection{Noise stability}
\label{subsec:noise-stability}
Let noisy observations satisfy
$\norm{y^\delta-y^\dagger}_{L^2}+\norm{u^\delta-u^\dagger}_U\le\delta$, and define the
data misfit and the (possibly regularized) objective
\[
  M_\delta(x):=\frac12\norm{y_x-y^\delta}_{L^2}^2+\frac\beta2\norm{u_x-u^\delta}_U^2,
  \qquad
  \mathcal J_\delta(x):=M_\delta(x)+R_\alpha(x),
\]
and write $M_0$ for the misfit with the exact data $(y^\dagger,u^\dagger)$, so that
$M_0=\mathcal J_0$ and $M_0(x^\dagger)=0$.

\begin{proposition}[Local noise stability, gradient-free]
\label{prop:noise-stability}
Assume the second-order growth condition
\begin{equation}
\label{eq:soc-growth}
  M_0(x)\ge M_0(x^\dagger)+\gamma\norm{x-x^\dagger}^2\qquad\text{for }x\in\Xad\text{ near }x^\dagger,
\end{equation}
with $\gamma>0$, and the regularization-bias condition: there is $C_R\ge0$,
independent of $\delta$, such that for $x$ near $x^\dagger$,
\begin{equation}
\label{eq:reg-bias}
  R_\alpha(x^\dagger)-R_\alpha(x)\le C_R\,\delta\,\norm{x-x^\dagger}.
\end{equation}
Then every local minimizer $x^\delta$ of $\mathcal J_\delta$ sufficiently close to
$x^\dagger$ satisfies $\norm{x^\delta-x^\dagger}\le C\delta$ for all sufficiently small
$\delta$, with $C$ independent of $\delta$. Condition \eqref{eq:reg-bias} holds
trivially with $C_R=0$ in the unregularized case, and holds with a
$\delta$-independent $C_R$ whenever $R_\alpha=\alpha R$ with $R\in C^1$ near
$x^\dagger$ and $\alpha=O(\delta)$ (then $C_R=\sup\norm{\nabla R}\cdot(\alpha/\delta)$
is bounded).
\end{proposition}

\begin{proof}
Local optimality of $x^\delta$ gives $\mathcal J_\delta(x^\delta)\le\mathcal
J_\delta(x^\dagger)$, i.e.
\begin{equation}
\label{eq:ns1}
  M_\delta(x^\delta)-M_\delta(x^\dagger)\le R_\alpha(x^\dagger)-R_\alpha(x^\delta)
  \le C_R\,\delta\,\norm{x^\delta-x^\dagger},
\end{equation}
using \eqref{eq:reg-bias}. Set $g(x):=M_0(x)-M_\delta(x)$. Expanding the squares,
\[
  g(x)=\tfrac12\inner{y^\delta-y^\dagger}{\,2y_x-y^\dagger-y^\delta}_{L^2}
       +\tfrac\beta2\inner{u^\delta-u^\dagger}{\,2u_x-u^\dagger-u^\delta}_U,
\]
hence, after cancellation of the data-only terms,
\[
  g(x^\delta)-g(x^\dagger)
  =\inner{y^\delta-y^\dagger}{\,y_{x^\delta}-y_{x^\dagger}}_{L^2}
   +\beta\inner{u^\delta-u^\dagger}{\,u_{x^\delta}-u_{x^\dagger}}_U .
\]
By Cauchy--Schwarz, the noise bound, and Lipschitz continuity of the response map
(Theorem~\ref{thm:lipschitz-response}), and since $y_{x^\dagger}=y^\dagger$,
$u_{x^\dagger}=u^\dagger$,
\begin{equation}
\label{eq:ns2}
  \bigl|g(x^\delta)-g(x^\dagger)\bigr|
  \le\bigl(\norm{y^\delta-y^\dagger}_{L^2}+\beta\norm{u^\delta-u^\dagger}_U\bigr)
     \,C_\Psi\norm{x^\delta-x^\dagger}
  \le (1+\beta)C_\Psi\,\delta\,\norm{x^\delta-x^\dagger}.
\end{equation}
Now
\[
  M_0(x^\delta)-M_0(x^\dagger)
  =\bigl[M_\delta(x^\delta)-M_\delta(x^\dagger)\bigr]+\bigl[g(x^\delta)-g(x^\dagger)\bigr],
\]
so combining \eqref{eq:ns1}, \eqref{eq:ns2}, and the growth condition
\eqref{eq:soc-growth} (recall $M_0(x^\dagger)=0$),
\[
  \gamma\norm{x^\delta-x^\dagger}^2
  \le M_0(x^\delta)-M_0(x^\dagger)
  \le \bigl(C_R+(1+\beta)C_\Psi\bigr)\,\delta\,\norm{x^\delta-x^\dagger}.
\]
If $x^\delta\neq x^\dagger$, divide by $\norm{x^\delta-x^\dagger}$ to obtain
$\norm{x^\delta-x^\dagger}\le C\delta$ with $C=\bigl(C_R+(1+\beta)C_\Psi\bigr)/\gamma$;
the inequality is trivial otherwise.
\end{proof}

\begin{remark}
The proof uses only the second-order growth of the noiseless misfit and Lipschitz
dependence of the response on the parameter; it does not require Fr\'echet
differentiability of $\Psi$, consistently with the nonsmoothness inherent to the
constrained problem. The local growth \eqref{eq:soc-growth} is, in turn, guaranteed
by the identifiability condition of Proposition~\ref{prop:identifiability} whenever
$\Psi$ is differentiable at $x^\dagger$, and may hold more generally.
\end{remark}

\section{Value-Function Reformulation and Exact Relaxation}
\label{sec:value}

\subsection{Value function}
\label{subsec:value-function}
Define the lower-level value function
$\varphi(x):=\min_{u\in\Uad}\Phi_x(u)=\Phi_x(u_x).$

\begin{theorem}[Concavity and Lipschitz continuity of the value function]
\label{thm:value-concave-lipschitz}
Under Assumptions~\ref{ass:A}--\ref{ass:C}, $\varphi:\Xad\to\R$ is concave. If in
addition the bounded-derivative hypothesis of Assumption~\ref{ass:B} holds (so that
$j(\Sop(\Uad))$ is bounded), then $\varphi$ is Lipschitz continuous, with constant
$L_\varphi:=\sup_{u\in\Uad}\norm{j(\Sop u)}_{\Rn}<\infty$. In all cases, for every
$x\in\Xad$, $j(y_x)$ is a concave supergradient:
\[
  \varphi(x')\le\varphi(x)+(x'-x)\cdot j(y_x)\qquad\forall x'\in\Xad .
\]
\end{theorem}

\begin{proof}
For fixed $u\in\Uad$, $x\mapsto x\cdot j(\Sop u)+\tfrac\sigma2\norm{u}_U^2$ is affine;
$\varphi$ is the pointwise infimum of affine functions, hence concave.

For the Lipschitz bound, note that $\Sop(\Uad)$ is bounded and convex in $Y$
(image of the bounded convex set $\Uad$ under the affine map $\Sop$). On a bounded
convex set, the bounded-derivative hypothesis of Assumption~\ref{ass:B} makes each
$j_\ell$ Lipschitz, by the mean-value inequality
\[
  |j_\ell(\Sop u)-j_\ell(\Sop u')|\le M_j\norm{\Sop u-\Sop u'}_Y,
\]
so in particular $j(\Sop(\Uad))$ is bounded; let
$L_\varphi:=\sup_{u\in\Uad}\norm{j(\Sop u)}_{\Rn}<\infty$. 
For any $x, x' \in \Xad$, let $u_x, u_{x'} \in \Uad$ be the unique minimizers associated with $\varphi(x)$ and $\varphi(x')$, respectively. Testing $u_{x'}$ as a suboptimal control for the problem corresponding to $x$ yields
\[
  \varphi(x) - \varphi(x') \le \Phi(x, u_{x'}) - \Phi(x', u_{x'}) = (x - x') \cdot j(\Sop u_{x'}).
\]
By interchanging the roles of $x$ and $x'$ and exploiting the boundedness of the mapping $u \mapsto j(\Sop u)$ over the bounded set $\Uad$, we readily obtain
\[
  |\varphi(x) - \varphi(x')| \le L_\varphi \norm{x - x'}_{\Rn}
\]
with $L_\varphi := \sup_{u \in \Uad} \norm{j(\Sop u)}_{\Rn} < \infty$.
Finally, $u_x$ is feasible for the lower-level problem at $x'$, so
\[
  \varphi(x')\le x'\cdot j(y_x)+\frac\sigma2\norm{u_x}_U^2
  =\varphi(x)+(x'-x)\cdot j(y_x).
\]
\end{proof}

\subsection{Optimal-value reformulation}
\label{subsec:ovr}
The bilevel problem is equivalent to the optimal-value reformulation
\begin{equation}
\tag{OVR}\label{eq:OVR}
\begin{aligned}
  \min_{x,y,u}\ &F(x,y,u)\\
  \text{s.t.}\ &x\in\Xad,\quad \Eop y-\Gop u=b_0,\quad u\in\Uad,\\
  &f(x,y,u)-\varphi(x)\le0,\qquad f(x,y,u):=x\cdot j(y)+\frac\sigma2\norm{u}_U^2 .
\end{aligned}
\end{equation}
At every feasible point the value constraint is active, so classical constraint
qualifications generally fail \cite{ye1995optimality,yu2026size}. The reformulation \eqref{eq:OVR} is
therefore useful for existence and relaxation arguments but is not the preferred form
for deriving stationarity; the KKT reformulation of Section~\ref{sec:mpcc} is used
for that purpose.

\subsection{Relaxed value-function problem}
\label{subsec:relaxed-value}
For $\varepsilon>0$, define
\begin{equation}
\tag{$\mathrm{OVR}_\varepsilon$}\label{eq:OVReps}
\begin{aligned}
  \min_{x,y,u}\ &F(x,y,u)\\
  \text{s.t.}\ &x\in\Xad,\quad \Eop y-\Gop u=b_0,\quad u\in\Uad,\quad
  f(x,y,u)-\varphi(x)\le\varepsilon.
\end{aligned}
\end{equation}

\begin{lemma}[Existence of relaxed minimizers]
\label{lem:ovreps-existence}
Under Assumptions~\ref{ass:A}--\ref{ass:D} with the bounded-derivative hypothesis of
Assumption~\ref{ass:B}, for each $\varepsilon\ge0$ the problem \eqref{eq:OVReps}
admits a global minimizer.
\end{lemma}

\begin{proof}
The feasible set is nonempty (it contains $(x,\Sop u_x,u_x)$ for any $x\in\Xad$,
which satisfies the value constraint with slack $0\le\varepsilon$). Let
$(x_j,y_j,u_j)$ be a minimizing sequence; eliminating the state via $y_j=\Sop u_j$,
the variables are $(x_j,u_j)\in\Xad\times\Uad$. Since $\Xad$ is compact and $\Uad$ is
weakly sequentially compact (Lemma~\ref{lem:box-basic}), after a subsequence
$x_j\to\bar x$ in $\Rn$ (strongly) and $u_j\rightharpoonup\bar u$ in $U$, hence
$y_j=\Sop u_j\rightharpoonup\bar y=\Sop\bar u$ in $Y$. The objective $F$ is weakly
lower semicontinuous (the misfit terms are convex and continuous, $R$ is weakly l.s.c.).

For the value constraint we treat $f(x,\Sop u,u)=x\cdot j(\Sop u)+\tfrac\sigma2
\norm{u}_U^2$ carefully, since $(x,u)\mapsto x\cdot j(\Sop u)$ is not
jointly convex. Split
\[
  x_j\cdot j(\Sop u_j)
  =\bar x\cdot j(\Sop u_j)+(x_j-\bar x)\cdot j(\Sop u_j).
\]
For the first term, $y\mapsto\bar x\cdot j(y)$ is convex and continuous (as
$\bar x\in\R^n_+$ and each $j_\ell$ is convex), hence weakly l.s.c., so
$\bar x\cdot j(\Sop\bar u)\le\liminf_j\bar x\cdot j(\Sop u_j)$. For the second term,
$\norm{(x_j-\bar x)\cdot j(\Sop u_j)}\le\norm{x_j-\bar x}\,\sup_{u\in\Uad}
\norm{j(\Sop u)}_{\Rn}\to0$, the supremum being finite by the bounded-derivative
hypothesis (cf.\ $L_\varphi$ in Theorem~\ref{thm:value-concave-lipschitz}). Adding the
weakly l.s.c.\ term $\tfrac\sigma2\norm{u}_U^2$ gives
$f(\bar x,\bar y,\bar u)\le\liminf_j f(x_j,y_j,u_j)$. Since $\varphi$ is continuous on
$\Xad$ (Theorem~\ref{thm:value-concave-lipschitz}) and $x_j\to\bar x$, we obtain
\[
  f(\bar x,\bar y,\bar u)-\varphi(\bar x)
  \le\liminf_j\bigl[f(x_j,y_j,u_j)-\varphi(x_j)\bigr]\le\varepsilon .
\]
Thus $(\bar x,\bar y,\bar u)$ is feasible and minimizing.
\end{proof}

\begin{theorem}[Exactness of the value relaxation in the limit]
\label{thm:value-relaxation-exact}
Suppose, in addition to Assumptions~\ref{ass:A}--\ref{ass:D} and the
bounded-derivative hypothesis, that \eqref{eq:IOC} is feasible. Let
$\varepsilon_k\downarrow0$ and let $(x_k,y_k,u_k)$ be global minimizers of
\eqref{eq:OVReps} with $\varepsilon=\varepsilon_k$ (which exist by
Lemma~\ref{lem:ovreps-existence}). Then there is a subsequence (not relabeled) and a
global minimizer $(\bar x,\bar y,\bar u)$ of \eqref{eq:IOC} with $x_k\to\bar x$ in
$\Rn$, $y_k\to\bar y$ in $Y$, and $u_k\to\bar u$ in $U$.
\end{theorem}

\begin{proof}
By compactness of $\Xad$, after a subsequence $x_k\to\bar x\in\Xad$. Feasibility gives
$f(x_k,y_k,u_k)-\varphi(x_k)\le\varepsilon_k$; since $y_k=\Sop u_k$ and
$\varphi(x_k)=\Phi_{x_k}(u_{x_k})$, this reads
$\Phi_{x_k}(u_k)-\Phi_{x_k}(u_{x_k})\le\varepsilon_k$. Because $u_{x_k}$ minimizes the
$\sigma$-strongly convex functional $\Phi_{x_k}$ over the convex set $\Uad$, the
strong-convexity inequality at the minimizer gives
\begin{equation}
\label{eq:sc-at-min}
  \Phi_{x_k}(u)\ge\Phi_{x_k}(u_{x_k})+\frac\sigma2\norm{u-u_{x_k}}_U^2
  \qquad\forall u\in\Uad,
\end{equation}
so in particular
\[
  \frac\sigma2\norm{u_k-u_{x_k}}_U^2\le\Phi_{x_k}(u_k)-\Phi_{x_k}(u_{x_k})\le\varepsilon_k,
\]
whence $u_k-u_{x_k}\to0$ in $U$. By Theorem~\ref{thm:lipschitz-response},
$u_{x_k}\to u_{\bar x}$ and $y_{x_k}\to y_{\bar x}$, hence $u_k\to u_{\bar x}$ and
$y_k=\Sop u_k\to\Sop u_{\bar x}=y_{\bar x}$. Thus $(\bar x,y_{\bar x},u_{\bar x})$ is
feasible for \eqref{eq:IOC}.

For optimality, any feasible point $(x,y,u)$ of \eqref{eq:IOC} is feasible for every
\eqref{eq:OVReps}, so $F(x_k,y_k,u_k)\le F(x,y,u)$. Lower semicontinuity of $F$ under
the established convergences yields $F(\bar x,\bar y,\bar u)\le F(x,y,u)$; hence
$(\bar x,\bar y,\bar u)$ is a global minimizer of \eqref{eq:IOC}.
\end{proof}

\begin{remark}
The relaxation in Theorem~\ref{thm:value-relaxation-exact} is a value-function
relaxation; it is distinct from the Scholtes relaxation of the complementarity
constraints used in Section~\ref{sec:mpcc}.
\end{remark}

\section{KKT Reformulation, MPCC Stationarity, and Scalar Strict Complementarity}
\label{sec:mpcc}

\subsection{Lower-level optimality system and equivalence}
\label{subsec:kkt-system}
We make the operator setting precise. The state operator acts as
\[
  \Eop:Y\to\Wsp:=L^2(0,T;\R^d)\times\R^d,\qquad
  \Eop y=(\dot y-\Aop y,\,y(0)),
\]
and $\Gop:U\to\Wsp$, $\Gop u=(\Bop u,0)$; the constraint is $\Eop y-\Gop u=b_0$ in
$\Wsp$. We identify $\Wsp$ with its dual $\Wsp^*$, so the adjoint variable satisfies
$p\in\Wsp^*\cong\Wsp$, and the adjoints act as $\Eop^*:\Wsp^*\to Y^*$,
$\Gop^*:\Wsp^*\to U^*$.

\medskip
\emph{Reduced optimality.} The reduced problem \eqref{eq:Pred} is the minimization of
the convex, Fr\'echet-differentiable functional $\Phi_x$ over the closed convex set
$\Uad$. By convex subdifferential calculus for a smooth convex function plus the
indicator of a closed convex set (no constraint qualification is required for this
sum rule, since $\Phi_x$ is finite and continuous on all of $U$), $u_x$ is the global
minimizer if and only if
\begin{equation}
\label{eq:reduced-incl}
  0\in S^*j'(\Sop u_x)^*x+\sigma u_x+N_{\Uad}(u_x),
\end{equation}
which is necessary and sufficient by convexity. Equivalently, there is
$\lambda\in U^*\cong U$ with
\begin{equation}
\label{eq:reduced-mult}
  S^*j'(\Sop u_x)^*x+\sigma u_x+\lambda=0,\qquad \lambda\in N_{\Uad}(u_x).
\end{equation}

\emph{Adjoint variable.} Introduce $p\in\Wsp^*$ as the solution of the linear adjoint
equation associated with the state constraint,
\begin{equation}
\label{eq:adjoint-def}
  \Eop^*p=-\,j'(y_x)^*x,\qquad y_x=\Sop u_x.
\end{equation}
The state operator $\Eop:Y\to\Wsp$ is a Banach-space isomorphism: for every
$(f,\zeta)\in\Wsp=L^2(0,T;\R^d)\times\R^d$ the linear ODE $\dot y-\Aop y=f$,
$y(0)=\zeta$, has a unique solution $y\in Y$ depending boundedly on $(f,\zeta)$
(Lemma~\ref{lem:control-to-state}). Hence the adjoint $\Eop^*:\Wsp^*\to Y^*$ is also a
Banach-space isomorphism, and \eqref{eq:adjoint-def} has a unique solution
$p\in\Wsp^*$. Writing $S=(\text{state part of }\Eop^{-1})\,\Gop$, so that
$S^*=\Gop^*(\text{state part of }\Eop^{-1})^*$, equation \eqref{eq:adjoint-def} gives
the identity
\[
  S^*j'(y_x)^*x=-\,\Gop^*p,
\]
so \eqref{eq:reduced-mult} becomes the standard co-state form. Collecting these, the
lower-level optimality system reads
\begin{equation}
\label{eq:kkt}
  \Eop^*p+j'(y)^*x=0,\quad
  \sigma u-\Gop^*p+\lambda=0,\quad
  \lambda\in N_{\Uad}(u),\quad
  \Eop y-\Gop u=b_0 ,
\end{equation}
with $p\in\Wsp^*$, $\lambda\in U$. (Formally, \eqref{eq:kkt} is the stationarity system
of the Lagrangian
$\mathscr L_{\mathrm{low}}(x,y,u,p)=x\cdot j(y)+\tfrac\sigma2\norm{u}_U^2
+\inner{p}{\Eop y-\Gop u-b_0}_{\Wsp}$.)

\begin{lemma}[Equivalence of the optimality system]
\label{lem:kkt-equiv}
Under Assumptions~\ref{ass:A}--\ref{ass:C}, for each $x\in\Xad$ a tuple $(y,u,p,\lambda)$
solves \eqref{eq:kkt} if and only if $(y,u)=\Psi(x)$ is the unique global solution of
\eqref{eq:Px}. Consequently \eqref{eq:IOC} is equivalent to the single-level problem
obtained by replacing $(y,u)\in\Psi(x)$ with the constraints \eqref{eq:kkt}.
\end{lemma}

\begin{proof}
The reduced inclusion \eqref{eq:reduced-incl} characterizes the global minimizer of
the convex problem \eqref{eq:Pred} without any constraint qualification,
because it is the optimality condition for minimizing the continuous convex function
$\Phi_x$ over the closed convex set $\Uad$ (equivalently, $0\in\partial\Phi_x(u_x)+
N_{\Uad}(u_x)$, and the subdifferential sum rule applies as $\Phi_x$ is
finite-valued and continuous). By Theorem~\ref{thm:lower-existence-uniqueness} the
minimizer $u_x$ is unique. The passage between \eqref{eq:reduced-incl} and
\eqref{eq:kkt} is the introduction of the adjoint $p$ via \eqref{eq:adjoint-def},
which is an equivalence since $\Eop^*$ is injective on $\Wsp^*$; thus \eqref{eq:kkt}
holds at $(y_x,u_x)$ with a unique $p$ and the corresponding $\lambda$, and
conversely any solution of \eqref{eq:kkt} yields \eqref{eq:reduced-mult} and hence
$u=u_x$. The equivalence of the two single-level problems follows.
\end{proof}

For box constraints, $N_{\Uad}(u)$ is characterized pointwise
(Lemma~\ref{lem:box-normal}): a.e.\ and componentwise, $\lambda(t)\le0$ on
$\{u=u_a\}$, $\lambda(t)=0$ on $\{u_a<u<u_b\}$, $\lambda(t)\ge0$ on $\{u=u_b\}$.
Writing $\lambda=\lambda^b-\lambda^a$ with $\lambda^a,\lambda^b\ge0$ and
$a:=u-u_a$, $b:=u_b-u$, this is equivalent to
\begin{equation}
\label{eq:cc}
  a\ge0,\ \lambda^a\ge0,\ a\lambda^a=0,\qquad
  b\ge0,\ \lambda^b\ge0,\ b\lambda^b=0,
\end{equation}
a.e.\ and componentwise. By Lemma~\ref{lem:kkt-equiv}, the inverse problem is thus an
MPCC in function space.

\subsection{Scholtes relaxation}
\label{subsec:scholtes}
For $\varepsilon>0$, the Scholtes relaxation replaces the equalities in \eqref{eq:cc}
by
\begin{equation}
\label{eq:scholtes}
  a\lambda^a\le\varepsilon,\qquad b\lambda^b\le\varepsilon,\qquad
  a,b,\lambda^a,\lambda^b\ge0,
\end{equation}
giving a family of relaxed nonlinear programs. This relaxation is used only for the
stationarity analysis, not for the value-function exactness result of
Section~\ref{sec:value}.

\begin{remark}[Function-space status of the relaxation]
\label{rem:scholtes-fnspace}
The pointwise products $a\lambda^a$, $b\lambda^b$ in \eqref{eq:scholtes} are delicate
in $L^2$: a product of two $L^2$ functions is in general only in $L^1$, so the smooth
nonlinear-programming calculus underlying the relaxed KKT system is not automatic in
the $L^2$ setting. Two standard remedies make \eqref{eq:scholtes} a well-posed family
of regular programs: (i) impose $L^\infty$ regularity on the relevant quantities --
e.g.\ $\lambda^a,\lambda^b\in L^\infty$, which holds when $j$ and the data are
sufficiently regular -- so that the complementarity map $(a,\lambda^a)\mapsto
a\lambda^a$ is well defined and differentiable into $L^2$; or (ii) treat the relaxed
KKT system itself as a hypothesis, as we do in Assumption~\ref{ass:mpcc-mfcq}. We
adopt (ii): the stationarity theorem below is formulated under the assumed existence
of relaxed KKT systems with multipliers in the stated function spaces. This is the
principal structural hypothesis of the C-stationarity result, and we flag it as the
main point requiring case-by-case verification in applications.
\end{remark}

\subsection{C-stationarity}
\label{subsec:cstationarity}
For a fixed representative of $\bar u$, define (up to null sets) the active and
inactive sets
\[
  I^a:=\{t:\bar u(t)=u_a(t)\},\quad
  I^b:=\{t:\bar u(t)=u_b(t)\},\quad
  I^0:=\{t:u_a(t)<\bar u(t)<u_b(t)\},
\]
and let the biactive set be
$\mathcal B:=\{t\in I^a\cup I^b:\ \bar\lambda(t)=0\}$, all understood a.e.\ and
componentwise. Figure~\ref{fig:control_trajectory} shows three subsets $I^a$, $I^b$ and $I^0$.

\begin{figure}[htbp]
    \centering
    \includegraphics[width=0.8\textwidth]{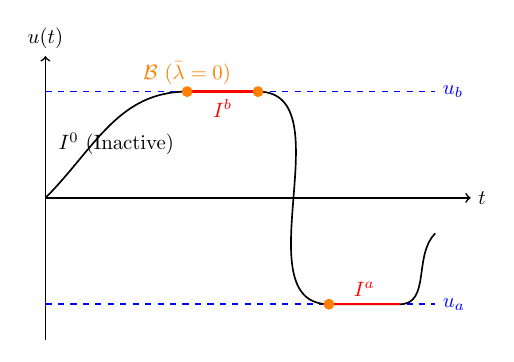}
    \caption{Illustration of a lower-level optimal control trajectory $u(t)$ subject to box constraints $u_a$ and $u_b$. The domain is partitioned into the inactive set $I^0$ where the control is strictly within the bounds, and the strongly active sets $I^a$ and $I^b$ where the control hits the lower and upper bounds, respectively. The orange markers denote the biactive points $\mathcal{B}$, where the control contacts the threshold with a vanishing multiplier ($\bar{\lambda}=0$).}
    \label{fig:control_trajectory}
\end{figure}

\begin{definition}[C-stationarity]
\label{def:cstationarity}
A feasible point $(\bar x,\bar y,\bar u)$ of \eqref{eq:IOC} is C-stationary if
there exist multipliers $\bar z\in\R^n$, lower-level and upper-level adjoints
$\bar p,\bar\eta\in\Wsp^*$, a linearized state $\bar\mu\in Y$, a linearized control
$\bar w\in U$, and $\bar\xi\in U$ such that the lower-level system \eqref{eq:kkt}
holds at $(\bar x,\bar y,\bar u,\bar p,\bar\lambda)$ and
\[
  0=F_x(\bar x,\bar y,\bar u)+\bar z+j'(\bar y)\bar\mu,\qquad
  0=F_y(\bar x,\bar y,\bar u)+\Eop^*\bar\eta+j''(\bar y)[\bar\mu]\bar x,
\]
\[
  0=F_u(\bar x,\bar y,\bar u)+\sigma\bar w-\Gop^*\bar\eta+\bar\xi,\qquad
  \bar z\in N_{\Xad}(\bar x),\qquad \Eop\bar\mu-\Gop\bar w=0,
\]
together with the complementarity sensitivity conditions, a.e.\ and componentwise,
\[
  \bar\xi=0\ \text{on }I^0,\qquad
  \bar w=0\ \text{on }\{\bar\lambda\neq0\},\qquad
  \bar\xi(t)\bar w(t)\ge0\ \text{on }\mathcal B .
\]
\end{definition}

\begin{remark}
This is C-stationarity, not strong stationarity. Strong stationarity would require
additional sign information ($\bar\xi\ge0$ and $\bar w$ of a definite sign) on the
biactive set $\mathcal B$, which Scholtes relaxation does not deliver in general; see
the stationarity hierarchy in Appendix~\ref{app:mpcc-background} and
\cite{scheel2000mathematical,hoheisel2013theoretical,yu2026pattern}.
\end{remark}

\begin{assumption}[Existence and uniform boundedness of relaxed KKT multipliers]
\label{ass:mpcc-mfcq}
Let $\varepsilon_k\downarrow0$ and let local minimizers of the Scholtes-relaxed
problems \eqref{eq:scholtes} converge to a feasible point
$(\bar x,\bar y,\bar u,\bar p,\bar\lambda)$. We assume that, for each $k$, the relaxed
local minimizer satisfies a KKT system in the stated function spaces, with multipliers
\[
  z_{\varepsilon_k}\in\R^n,\quad
  \eta_{\varepsilon_k}\in\Wsp^*,\quad
  \mu_{\varepsilon_k}\in Y,\quad
  w_{\varepsilon_k},\xi_{\varepsilon_k}\in U=L^2,\quad
  \rho^a_{\varepsilon_k},\rho^b_{\varepsilon_k}\in L^2(0,T;\R^m)_{\ge0},
\]
and that these multipliers are bounded uniformly in $k$. Here $\bar p,\eta\in
\Wsp^*\cong\Wsp$ are the lower- and upper-level adjoints associated with the state
constraint $\Eop y-\Gop u=b_0$ (so that $\Eop^*\eta\in Y^*$, $\Gop^*\eta\in U^*$),
and $\rho^a,\rho^b$ are the nonnegative multipliers of the relaxed complementarity
inequalities in \eqref{eq:scholtes}.
\end{assumption}

\begin{remark}[On the underlying constraint qualification]
\label{rem:mpcc-mfcq}
Assumption~\ref{ass:mpcc-mfcq} bundles two distinct hypotheses, neither of which is a
property of the Scholtes relaxation itself: (a) the relaxed local minimizers admit KKT
multipliers in the stated function spaces, and (b) those multipliers are bounded
uniformly in $k$. Even though the relaxation produces a regular NLP for each fixed
$\varepsilon_k$ in finite dimensions, in the $L^2$ setting the existence of KKT
multipliers requires the function-space calculus discussed in
Remark~\ref{rem:scholtes-fnspace}, and the uniform bound (b) requires a constraint
qualification -- an MPCC--MFCQ in finite dimensions, or its function-space analogues
for control-constrained MPECs \cite{hintermuller2014several,mehlitz2016weak,deng2021coupled,yu2026age}.
In the decoupled scalar setting of Section~\ref{subsec:scalar-sc} the
bound-complementarity strict-complementarity component of such a qualification holds
automatically (Corollary~\ref{cor:decoupled-two-agent-sc-app}).
\end{remark}

\begin{theorem}[C-stationarity of relaxation limits]
\label{thm:cstationarity}
Let $(\bar x,\bar y,\bar u)$ be a local minimizer of \eqref{eq:IOC}, and let
$\{(x_{\varepsilon_k},y_{\varepsilon_k},u_{\varepsilon_k})\}$ be a sequence of local
minimizers of the Scholtes-relaxed problems \eqref{eq:scholtes} with
$\varepsilon_k\downarrow0$ that converges to $(\bar x,\bar y,\bar u)$ (together with
the associated lower-level multipliers). Under
Assumptions~\ref{ass:A}--\ref{ass:D} and the uniform multiplier boundedness
condition of Assumption~\ref{ass:mpcc-mfcq}, the limit point
$(\bar x,\bar y,\bar u)$ is C-stationary in the sense of
Definition~\ref{def:cstationarity}.
\end{theorem}

\begin{proof}
By Lemma~\ref{lem:kkt-equiv}, near the local minimizer $(\bar x,\bar y,\bar u)$ the
problem \eqref{eq:IOC} is equivalent to the single-level KKT-constrained problem --
this equivalence rests on convexity of the lower level, which makes \eqref{eq:kkt}
sufficient as well as necessary. The relaxed local minimizers
$(x_{\varepsilon_k},y_{\varepsilon_k},u_{\varepsilon_k})$ of the localized Scholtes
relaxation \eqref{eq:scholtes} are given by hypothesis.

\emph{Step 1: relaxed KKT.} By Assumption~\ref{ass:mpcc-mfcq}, for each $\varepsilon_k$
the relaxed local minimizer satisfies a KKT system with multipliers
$(z_{\varepsilon_k},\eta_{\varepsilon_k},\mu_{\varepsilon_k},w_{\varepsilon_k},
\xi_{\varepsilon_k},\rho^a_{\varepsilon_k},\rho^b_{\varepsilon_k})$ in the stated
spaces, bounded uniformly in $k$. Here $\xi_{\varepsilon_k}$ collects the
bound-multiplier sensitivities and $w_{\varepsilon_k}$ the primal (control)
sensitivity, and the relaxed complementarity contributes, through \eqref{eq:scholtes},
terms of the form $\rho^a_{\varepsilon_k}\lambda^a$ and $\rho^a_{\varepsilon_k}a$ (and
the analogues for $b$) into the stationarity equations for $a$ and $\lambda^a$.

\emph{Step 2: weak limits.} By Assumption~\ref{ass:mpcc-mfcq} the multipliers are
uniformly bounded, so after extracting a subsequence,
$z_{\varepsilon_k}\to\bar z$ in $\R^n$, $\eta_{\varepsilon_k}\rightharpoonup\bar\eta$ in
$\Wsp^*$, $\mu_{\varepsilon_k}\rightharpoonup\bar\mu$ in $Y$, and
$w_{\varepsilon_k}\rightharpoonup\bar w$, $\xi_{\varepsilon_k}\rightharpoonup\bar\xi$ in
$L^2$. Passing to the limit in the (linear) stationarity equations gives
\[
  0=F_x+\bar z+j'(\bar y)\bar\mu,\quad
  0=F_y+\Eop^*\bar\eta+j''(\bar y)[\bar\mu]\bar x,\quad
  0=F_u+\sigma\bar w-\Gop^*\bar\eta+\bar\xi,
\]
with $\bar z\in N_{\Xad}(\bar x)$ (closedness of the finite-dimensional normal cone)
and $\Eop\bar\mu-\Gop\bar w=0$. The lower-level KKT system \eqref{eq:kkt} passes to
the limit by closedness of the box normal-cone relations under the a.e.\ pointwise
characterization of Lemma~\ref{lem:box-normal}.

\emph{Step 3: sign conditions.}
\begin{itemize}
\item \emph{Inactive set.} Fix $\delta>0$ and consider the localized inactive set
$I^0_\delta:=\{t:\ u_a(t)+\delta\le\bar u(t)\le u_b(t)-\delta\}$. On $I^0_\delta$ both
gaps satisfy $a,b\ge\delta>0$. The relaxed feasibility $a\lambda^a\le\varepsilon_k$,
$b\lambda^b\le\varepsilon_k$ then forces $\lambda^a\le\varepsilon_k/\delta\to0$ and
$\lambda^b\le\varepsilon_k/\delta\to0$, so the limiting lower-level multiplier
satisfies $\bar\lambda=0$ a.e.\ on $I^0_\delta$, consistent with the normal-cone
characterization (Lemma~\ref{lem:box-normal}). For the complementarity-sensitivity
multiplier $\bar\xi$, we do not assert that $\rho^a_{\varepsilon_k}$, $\rho^b_{\varepsilon_k}$
vanish (strict slackness is not guaranteed); instead $\bar\xi=0$ a.e.\ on the inactive
set is the inactive-index ($I^{+0}\cup I^{0+}$-type) component of the limiting
Lebesgue-space MPCC stationarity system, which holds under the present uniform
multiplier bound by the normal-cone/sign analysis of \cite{mehlitz2018limiting,liu2026ftu}
together with the finite-dimensional relaxation limits of \cite{scholtes2001convergence,hoheisel2013theoretical,liu2026computational}.
Letting $\delta\downarrow0$ and noting $I^0=\bigcup_{\delta>0}I^0_\delta$ up to a null
set gives $\bar\xi=0$ a.e.\ on $I^0$.
\item \emph{Strongly active set.} On $\{\bar\lambda\neq0\}$ (say $\bar\lambda>0$, the
upper bound) the constraint $b=u_b-u$ stays active with a nonvanishing multiplier;
the relaxed stationarity then forces the primal control sensitivity to vanish,
$\bar w=0$ a.e.\ there.
\item \emph{Biactive set.} On $\mathcal B$ ($a\to0$ and $\lambda^a\to0$, resp.\ $b,
\lambda^b\to0$) the Scholtes sign mechanism applies. With $\rho^a_{\varepsilon_k}\ge0$,
the two stationarity contributions that survive into the limit are linked through the
nonnegative quantity $\rho^a_{\varepsilon_k}\,a\,\lambda^a\ge0$; consequently the
limiting bound-sensitivity $\bar\xi$ and control-sensitivity $\bar w$ inherit a common
sign, i.e.
\[
  \bar\xi(t)\,\bar w(t)\ge0\qquad\text{a.e.\ on }\mathcal B.
\]
The sign relation is the conclusion of the standard relaxation-limit argument: the
finite-dimensional Scholtes multiplier bookkeeping
\cite{scholtes2001convergence,hoheisel2013theoretical,yu2026chemotactic} yields the Clarke
product sign on the biactive set, and its pointwise a.e.\ transfer to the
Lebesgue-space setting uses the variational geometry of pointwise-defined
complementarity sets in \cite{mehlitz2018limiting,yu2026fullcovariance} (extraction of an a.e.\ convergent
subsequence and Mazur's lemma to preserve the sign in the weak limit). We invoke these
results rather than reproving the full function-space statement here.
\end{itemize}
These are exactly the conditions of Definition~\ref{def:cstationarity}.
\end{proof}

\begin{remark}
We do not assert that every local minimizer of \eqref{eq:IOC} is
C-stationary. The conclusion is for limits of Scholtes-relaxed local minimizers
carrying uniformly bounded multipliers; whether an arbitrary local minimizer arises
as such a limit is a separate question, governed by the constraint qualification
discussed in Remark~\ref{rem:mpcc-mfcq}. This is the standard scope of
relaxation-based stationarity results
\cite{scholtes2001convergence,hoheisel2013theoretical,mehlitz2016weak,wang2025multi}.
\end{remark}

\subsection{Strict complementarity in the decoupled scalar case}
\label{subsec:scalar-sc}
We record a scalar result useful for checking the MPCC constraint qualification in a
transparent setting. Consider
\[
  \min_{u\in[u_a,u_b]}\int_0^T\Bigl(cy(t)^2+\frac\sigma2 u(t)^2\Bigr)dt
  \quad\text{s.t.}\quad \dot y=ay+bu,\ y(0)=y_0,
\]
with $a\in\R$, $b\neq0$, $c\ge0$, $\sigma>0$, $u_a<0<u_b$.

\begin{theorem}[Unconditional finiteness and a.e.\ strict complementarity]
\label{thm:scalar-sc}
For this decoupled scalar constrained LQ problem, the threshold-contact set
$\{\rho=u_a\}\cup\{\rho=u_b\}$ is finite. Consequently the control can enter or leave
a bound only at the finitely many threshold-contact times, and the biactive set has
Lebesgue measure zero, so strict complementarity holds almost everywhere, in the
sense relevant for the $L^2$ normal cone. No nondegeneracy condition is required.
Hence, in the two-agent model of Section~\ref{subsec:two-agent} with $x_1=0$, strict
complementarity holds a.e.\ for every admissible datum.
\end{theorem}

\begin{proof}
See Appendix~\ref{app:scalar-sc} (Theorem~\ref{thm:scalar-sc-app} and
Corollary~\ref{cor:decoupled-two-agent-sc-app}). In outline: the switching function
$\rho:=-(b/\sigma)p$ is $C^2$; if $c=0$ then $p\equiv0$ and no threshold contact
occurs; if $c>0$, an accumulation point of a threshold-contact set $\{\rho=u_a\}$ or
$\{\rho=u_b\}$ would force, via a twofold Rolle argument, $\rho=u_b$, $\dot\rho=0$,
$\ddot\rho=0$ simultaneously, which overdetermines the state--costate system and
yields the impossible identity $a^2\sigma+2b^2c=0$. Hence the threshold-contact set is
finite; the control can enter or leave a bound only at these finitely many times, and
since the biactive set $\mathcal B$ is contained in the threshold-contact set,
$\operatorname{meas}(\mathcal B)=0$.
\end{proof}

\appendix

\section{Auxiliary Hilbert-Space Facts}
\label{app:hilbert}
This appendix collects the elementary functional-analytic facts used in
Sections~\ref{sec:wellposed}--\ref{sec:mpcc}, making explicit where boundedness of the
admissible control set enters the direct method.

\subsection{Weak compactness of the admissible controls}
With $U=L^2(0,T;\R^m)$ and $\Uad=\{u:\ u_a\le u\le u_b\ \text{a.e.}\}$,
$u_a,u_b\in L^\infty$, $u_a\le u_b$ a.e.:

\begin{lemma}[Basic properties of the box constraint]
\label{lem:box-basic}
$\Uad$ is nonempty, closed, convex, and bounded in $U$, and hence weakly sequentially
compact.
\end{lemma}

\begin{proof}
Convexity is immediate from the componentwise interval constraints. Closedness: if
$u_k\to u$ in $L^2$, a subsequence converges a.e., and the pointwise inequalities pass
to the limit. Boundedness: every $u\in\Uad$ satisfies
$|u(t)|\le|u_a(t)|+|u_b(t)|$ a.e., so
$\norm{u}_{L^2}\le T^{1/2}(\norm{u_a}_{L^\infty}+\norm{u_b}_{L^\infty})$. Since $U$ is a
Hilbert space, it is reflexive, and a bounded closed convex set in a reflexive Banach
space is weakly sequentially compact.
\end{proof}

\subsection{Control-to-state map}
\begin{lemma}[Affine control-to-state map]
\label{lem:control-to-state}
For every $u\in U$ there is a unique $y_u\in Y=H^1(0,T;\R^d)$ solving the state
equation, and $\Sop:u\mapsto y_u$ is affine and continuous: $\Sop u=y^0+Su$ with
$S\in\Lc(U,Y)$.
\end{lemma}

\begin{proof}
By variation of constants,
$y_u(t)=e^{t\Aop}y_0+\int_0^t e^{(t-s)\Aop}\Bop u(s)\,ds.$ The first term defines
$y^0\in H^1$, the second a linear map $S:U\to Y$, and standard linear-ODE estimates
give $\norm{Su}_{H^1}\le C_{\Aop,\Bop,T}\norm{u}_{L^2}$, so $S\in\Lc(U,Y)$.
\end{proof}

\subsection{Weak lower semicontinuity}
\begin{lemma}[Weak l.s.c.\ of convex continuous functionals]
\label{lem:wlsci}
If $X$ is a Hilbert space and $\phi:X\to\R$ is convex and strongly continuous, then
$\phi$ is weakly sequentially lower semicontinuous.
\end{lemma}

\begin{proof}
Each sublevel set $\{\phi\le\alpha\}$ is strongly closed and convex, hence weakly
closed; weak closedness of all sublevel sets is equivalent to weak sequential lower
semicontinuity.
\end{proof}

\subsection{The box normal cone in $L^2$}
For closed convex $K$ in a Hilbert space $H$,
$N_K(\bar u):=\{\lambda:\ \inner{\lambda}{v-\bar u}_H\le0\ \forall v\in K\}.$

\begin{lemma}[Normal cone to the box]
\label{lem:box-normal}
Let $\bar u\in\Uad$. Then $\lambda\in N_{\Uad}(\bar u)$ iff, componentwise and a.e.,
$\lambda_i\le0$ on $\{\bar u_i=u_{a,i}\}$, $\lambda_i=0$ on
$\{u_{a,i}<\bar u_i<u_{b,i}\}$, $\lambda_i\ge0$ on $\{\bar u_i=u_{b,i}\}$.
Equivalently, with $a:=\bar u-u_a$, $b:=u_b-\bar u$, $\lambda=\lambda^b-\lambda^a$,
$\lambda^a,\lambda^b\ge0$, one has $a\lambda^a=0$, $b\lambda^b=0$ a.e.\ componentwise.
\end{lemma}

\begin{proof}
Pointwise: if $\bar u_i(t)$ is interior, both signs of perturbation are admissible,
forcing $\lambda_i(t)=0$; at the lower bound only nonnegative perturbations are
admissible, forcing $\lambda_i(t)\le0$; at the upper bound only nonpositive ones,
forcing $\lambda_i(t)\ge0$. Integrating gives the Hilbert-space characterization. The
complementarity form follows by setting $\lambda=\lambda^b-\lambda^a$.
\end{proof}

\section{MPCC Stationarity Background}
\label{app:mpcc-background}
We recall the stationarity notions for MPCCs, in the pointwise form relevant to box
control constraints, following
\cite{scheel2000mathematical,scholtes2001convergence,luo1996mathematical,hoheisel2013theoretical,gao2022rolling}.

\subsection{Prototype complementarity structure}
A finite-dimensional MPCC contains
$0\le G_i(z)\perp H_i(z)\ge0$, $i=1,\dots,q$, i.e.\ $G_i\ge0$, $H_i\ge0$,
$G_iH_i=0$. For the control box the two pairs are
$a=u-u_a\ge0$, $\lambda^a\ge0$, $a\lambda^a=0$ and
$b=u_b-u\ge0$, $\lambda^b\ge0$, $b\lambda^b=0$, with index sets at a feasible point
\[
  I^{+0}:=\{G>0,H=0\},\quad I^{0+}:=\{G=0,H>0\},\quad I^{00}:=\{G=0,H=0\},
\]
understood a.e.\ and componentwise.

\subsection{Stationarity hierarchy}
Let $(\alpha,\beta)$ denote the multipliers of a complementarity pair $(G,H)$, with
the remaining smooth constraints suppressed.

\begin{definition}[Weak stationarity]
\label{def:weak-stationarity-app}
Lagrangian stationarity holds and $\alpha=0$ on $I^{+0}$, $\beta=0$ on $I^{0+}$; no
sign condition on $I^{00}$.
\end{definition}
\begin{definition}[C-stationarity]
\label{def:c-stationarity-app}
Weakly stationary and, in addition, $\alpha_i\beta_i\ge0$ on $I^{00}$ (a.e.,
componentwise in the function-space setting).
\end{definition}
\begin{definition}[M-stationarity]
\label{def:m-stationarity-app}
Weakly stationary and, for every $i\in I^{00}$, either $\alpha_i\beta_i=0$ or
$\alpha_i>0,\beta_i>0$.
\end{definition}
\begin{definition}[Strong stationarity]
\label{def:strong-stationarity-app}
Weakly stationary and, for every $i\in I^{00}$, $\alpha_i\ge0$ and $\beta_i\ge0$.
\end{definition}

These satisfy
\[
  \text{strong}\Rightarrow\text{M}\Rightarrow\text{C}\Rightarrow\text{weak},
\]
with strict implications in general.

\subsection{Relation to the box-control MPCC}
The MPCC degeneracy for the box control occurs where a control is at a bound and the
corresponding bound multiplier vanishes -- the biactive set
$\mathcal B=\{t\in I^a(\bar u)\cup I^b(\bar u):\bar\lambda(t)=0\}$. The conditions of
Definition~\ref{def:cstationarity} are the function-space analogue of
Definition~\ref{def:c-stationarity-app}: $\bar\xi=0$ on $I^0(\bar u)$, $\bar w=0$ on
$\{\bar\lambda\neq0\}$, and $\bar\xi\bar w\ge0$ on $\mathcal B$.

\subsection{Why C-stationarity is the natural relaxation limit}
Scholtes regularization replaces $G_iH_i=0$ by $G_iH_i\le\varepsilon$. For each fixed
$\varepsilon>0$ the relaxed problem is an ordinary NLP, and standard KKT theory
applies under a usual CQ. As $\varepsilon\downarrow0$, multiplier limits generally
satisfy only the Clarke sign condition on the biactive set
\cite{scholtes2001convergence,hoheisel2013theoretical}; hence C-stationarity is the robust notion obtained
without further structure. Strong stationarity requires stronger assumptions, e.g.\
absence of biactive indices, strict complementarity, or MPCC-tailored CQs. This is
why Theorem~\ref{thm:cstationarity} asserts C-stationarity for relaxation limits
under a uniform multiplier bound, while Theorem~\ref{thm:scalar-sc} gives a separate
structural condition under which the biactive set has measure zero and strict
complementarity holds a.e.

\section{Strict Complementarity in the Decoupled Scalar Case}
\label{app:scalar-sc}
This appendix proves the strict-complementarity assertion used in
Section~\ref{subsec:scalar-sc}. The result is deliberately stated in the
almost-everywhere sense appropriate for box constraints in $L^2$: isolated
switching-threshold contacts are invisible to the $L^2$-normal cone, and hence do
not obstruct the MPCC qualification at the level of Lebesgue-space complementarity.

\subsection{Scalar constrained LQ problem}
Fix
\[
  a\in\R,\qquad b\neq0,\qquad c\ge0,\qquad \sigma>0,\qquad u_a<0<u_b,
\]
and consider the scalar constrained linear-quadratic problem
\begin{equation}
\tag{SP}\label{eq:scalar-problem-app}
  \min_{u\in\Uad}
  \int_0^T\left(cy(t)^2+\frac{\sigma}{2}u(t)^2\right)\,dt,
  \qquad
  \dot y=ay+bu,\quad y(0)=y_0,
\end{equation}
where
\[
  \Uad:=\{u\in L^2(0,T): u_a\le u(t)\le u_b\ \text{for a.e. }t\in(0,T)\}.
\]
The admissible set is nonempty, closed, convex, and bounded in $L^2(0,T)$, and the
reduced objective is strongly convex because $\sigma>0$. Hence
\eqref{eq:scalar-problem-app} has a unique optimal control, denoted $\bar u$, and a
unique associated state, denoted $\bar y$.

With Hamiltonian
\[
  H(y,u,p):=cy^2+\frac{\sigma}{2}u^2+p(ay+bu),
\]
the first-order optimality system is
\begin{equation}
\label{eq:scalar-os-app}
  \dot{\bar y}=a\bar y+b\bar u,\qquad \bar y(0)=y_0,
  \qquad
  \dot p=-ap-2c\bar y,\qquad p(T)=0,
\end{equation}
together with the pointwise variational inequality
\begin{equation}
\label{eq:scalar-vi-app}
  \bigl(\sigma\bar u(t)+bp(t)\bigr)\bigl(v-\bar u(t)\bigr)\ge0
  \qquad
  \forall v\in[u_a,u_b],
  \quad\text{for a.e. }t\in(0,T).
\end{equation}
Equivalently,
\begin{equation}
\label{eq:scalar-projection-app}
  \bar u(t)=\proj_{[u_a,u_b]}(\rho(t)),
  \qquad
  \rho(t):=-\frac b\sigma p(t).
\end{equation}
The conditions \eqref{eq:scalar-os-app}--\eqref{eq:scalar-projection-app} are not
only necessary but also sufficient, since the problem is convex and the reduced
objective is strongly convex.

We call $\rho$ the switching function. The active-bound regions are
\[
  \{\bar u=u_b\}=\{\rho\ge u_b\},
  \qquad
  \{\bar u=u_a\}=\{\rho\le u_a\},
\]
where the equalities are understood pointwise after taking the continuous
representatives established below. The set
\[
  \mathcal T:=\{\rho=u_a\}\cup\{\rho=u_b\}
\]
is called the threshold-contact set. It is this set, not the full active-bound
region, that can give rise to biactivity.

\subsection{Regularity of the switching function}
\begin{lemma}[Regularity bootstrap]
\label{lem:scalar-C2-app}
The optimality system has the following regularity:
\[
  \bar u\in C([0,T]),\qquad
  \bar y\in C^1([0,T]),\qquad
  p\in C^2([0,T]),\qquad
  \rho\in C^2([0,T]).
\]
At the endpoints the derivatives are understood one-sidedly.
\end{lemma}
\begin{proof}
Since $\bar u\in L^2(0,T)$, the state equation gives $\bar y\in H^1(0,T)$. The
adjoint equation then gives $p\in H^1(0,T)$. In one dimension,
$H^1(0,T)\hookrightarrow C([0,T])$, so $p$ is continuous, and therefore
$\rho=-(b/\sigma)p$ is continuous. Since the projection onto $[u_a,u_b]$ is Lipschitz,
\eqref{eq:scalar-projection-app} implies $\bar u\in C([0,T])$.

Now the state equation has a continuous right-hand side $a\bar y+b\bar u\in C([0,T])$,
hence $\bar y\in C^1([0,T])$. Similarly, since $p,\bar y\in C([0,T])$, the adjoint
equation gives $\dot p=-ap-2c\bar y\in C([0,T])$, so $p\in C^1([0,T])$. Using now
$p,\bar y\in C^1([0,T])$, we obtain $-ap-2c\bar y\in C^1([0,T])$, and consequently
$p\in C^2([0,T])$. Thus $\rho=-\tfrac b\sigma p\in C^2([0,T])$.
\end{proof}

\subsection{No accumulation of threshold contacts}
\begin{lemma}[No accumulation of threshold contacts]
\label{lem:scalar-noaccum-app}
If $c>0$, neither $\{\rho=u_b\}$ nor $\{\rho=u_a\}$ has an accumulation point in
$[0,T]$. If $c=0$, then $\rho\equiv0$, and therefore
$\{\rho=u_a\}\cup\{\rho=u_b\}=\varnothing$.
\end{lemma}
\begin{proof}
If $c=0$, the adjoint equation reduces to $\dot p=-ap$, $p(T)=0$, so $p\equiv0$ and
$\rho\equiv0$. Since $u_a<0<u_b$, $\rho(t)\in(u_a,u_b)$ for every $t\in[0,T]$, and
there are no threshold contacts.

Assume $c>0$. We treat the upper threshold; the lower is identical. Suppose
$\{\rho=u_b\}$ has an accumulation point $t^*\in[0,T]$. Set $g:=\rho-u_b\in
C^2([0,T])$ (Lemma~\ref{lem:scalar-C2-app}). There are distinct $\tau_k\to t^*$ with
$g(\tau_k)=0$, so by continuity $g(t^*)=0$. If $t^*\in(0,T)$, Rolle between
consecutive zeros gives $\eta_k\to t^*$ with $\dot g(\eta_k)=0$, hence $\dot
g(t^*)=0$; a second application to $\dot g$ gives $\zeta_k\to t^*$ with $\ddot
g(\zeta_k)=0$, hence $\ddot g(t^*)=0$. If $t^*\in\{0,T\}$, the same conclusion holds
with one-sided derivatives: accumulating zeros at the endpoint force the
corresponding one-sided first derivative of $g$ to vanish, and Rolle between
consecutive zeros produces interior zeros of $\dot g$ accumulating at the endpoint,
forcing the corresponding one-sided second derivative to vanish. Thus in all cases
\[
  \rho(t^*)=u_b,\qquad \dot\rho(t^*)=0,\qquad \ddot\rho(t^*)=0,
\]
with one-sided interpretation at the endpoints. Since $\rho=-(b/\sigma)p$ and
$b\neq0$,
\[
  p(t^*)=P_b:=-\frac{\sigma u_b}{b},\qquad \dot p(t^*)=0,\qquad \ddot p(t^*)=0.
\]
Evaluating the adjoint equation at $t^*$ gives $0=\dot p(t^*)=-ap(t^*)-2c\bar
y(t^*)$, hence
\begin{equation}
\label{eq:scalar-y-star-app}
  \bar y(t^*)=-\frac{aP_b}{2c}=\frac{a\sigma u_b}{2bc}.
\end{equation}
Differentiating the adjoint equation, $\ddot p=-a\dot p-2c\dot{\bar y}$; using $\dot
p(t^*)=\ddot p(t^*)=0$,
\begin{equation}
\label{eq:scalar-ydot-zero-app}
  \dot{\bar y}(t^*)=0.
\end{equation}
On the other hand, $\rho(t^*)=u_b$ and \eqref{eq:scalar-projection-app} give $\bar
u(t^*)=u_b$, so by the state equation and \eqref{eq:scalar-y-star-app},
\[
  \dot{\bar y}(t^*)=a\bar y(t^*)+b\bar u(t^*)
  =a\frac{a\sigma u_b}{2bc}+bu_b=u_b\,\frac{a^2\sigma+2b^2c}{2bc}.
\]
Combined with \eqref{eq:scalar-ydot-zero-app} this forces
$u_b\,(a^2\sigma+2b^2c)/(2bc)=0$. Since $u_b>0$, $b\neq0$, $c>0$, $\sigma>0$, we get
$a^2\sigma+2b^2c=0$, impossible because $a^2\sigma\ge0$ and $2b^2c>0$. Hence
$\{\rho=u_b\}$ has no accumulation point. The lower-threshold case is identical,
replacing $u_b$ by $u_a$; since $u_a<0$, in particular $u_a\neq0$, the same
contradiction is obtained.
\end{proof}

\subsection{Finiteness and almost-everywhere strict complementarity}
We now fix the correct multiplier convention. The variational inequality
\eqref{eq:scalar-vi-app} is equivalent to
\[
  0\in \sigma\bar u(t)+bp(t)+N_{[u_a,u_b]}(\bar u(t))
  \qquad\text{for a.e. }t\in(0,T),
\]
where
\[
  N_{[u_a,u_b]}(s)=
  \begin{cases}
    (-\infty,0], & s=u_a,\\
    \{0\},       & u_a<s<u_b,\\
    [0,\infty),  & s=u_b.
  \end{cases}
\]
Thus, with the normal-cone sign convention of the main text, the scalar box
multiplier is
\begin{equation}
\label{eq:scalar-normal-mult-app}
  \lambda_N(t):=-\bigl(\sigma\bar u(t)+bp(t)\bigr)=\sigma\bigl(\rho(t)-\bar u(t)\bigr),
\end{equation}
so that
\[
  \lambda_N\le0\ \text{on }\{\bar u=u_a\},\qquad
  \lambda_N=0\ \text{on }\{u_a<\bar u<u_b\},\qquad
  \lambda_N\ge0\ \text{on }\{\bar u=u_b\}.
\]
In the notation of the main text \eqref{eq:cc}, $\lambda_N=\lambda^b-\lambda^a$ with
$\lambda^a:=(\lambda_N)^-\ge0$ and $\lambda^b:=(\lambda_N)^+\ge0$, so $\lambda_N$ is
exactly the scalar instance of the box normal-cone multiplier $\lambda$ of
\eqref{eq:kkt}. The biactive set is
\[
  \mathcal B:=\{t\in[0,T]:\bar u(t)\in\{u_a,u_b\},\ \lambda_N(t)=0\}.
\]

\begin{theorem}[Unconditional finiteness and a.e. strict complementarity]
\label{thm:scalar-sc-app}
For \eqref{eq:scalar-problem-app} with $a\in\R$, $b\neq0$, $c\ge0$, $\sigma>0$,
$u_a<0<u_b$, the threshold-contact set $\mathcal T=\{\rho=u_a\}\cup\{\rho=u_b\}$ is
finite. The control can enter or leave a bound only at threshold-contact times (an
active arc $\{\bar u=u_b\}=\{\rho\ge u_b\}$ may itself be an interval of positive
measure). Consequently $\mathcal B\subset\mathcal T$ is finite, and therefore
$\operatorname{meas}(\mathcal B)=0$ (Figure~\ref{fig:switching_function}); equivalently, strict complementarity holds almost everywhere in the $L^2$-normal-cone sense. No additional nondegeneracy condition is required.
\end{theorem}
\begin{proof}
If $c=0$, Lemma~\ref{lem:scalar-noaccum-app} gives $\rho\equiv0\in(u_a,u_b)$, so
$\mathcal T=\varnothing$. Let $c>0$. By Lemma~\ref{lem:scalar-noaccum-app}, neither
$\{\rho=u_a\}$ nor $\{\rho=u_b\}$ has an accumulation point in $[0,T]$. Since $[0,T]$
is compact, any infinite subset has an accumulation point (Bolzano--Weierstrass);
hence both sets are finite, so $\mathcal T$ is finite.

It remains to relate $\mathcal T$ to biactivity. By the projection formula
$\bar u(t)=\proj_{[u_a,u_b]}(\rho(t))$. If $u_a<\bar u(t)<u_b$, then $\bar
u(t)=\rho(t)$ and $\lambda_N(t)=0$, but no bound is active. If $\bar u(t)=u_b$, then
$\rho(t)\ge u_b$ and \eqref{eq:scalar-normal-mult-app} gives
$\lambda_N(t)=\sigma(\rho(t)-u_b)$, so $\lambda_N(t)=0$ iff $\rho(t)=u_b$. Likewise,
if $\bar u(t)=u_a$, then $\rho(t)\le u_a$ and $\lambda_N(t)=\sigma(\rho(t)-u_a)$, so
$\lambda_N(t)=0$ iff $\rho(t)=u_a$. Therefore $\mathcal B\subset\{\rho=u_a\}\cup
\{\rho=u_b\}=\mathcal T$. Since $\mathcal T$ is finite, $\mathcal B$ is finite, so
$\operatorname{meas}(\mathcal B)=0$, which is precisely a.e.\ strict complementarity
for the $L^2$-box normal cone.
\end{proof}

\begin{figure}
    \centering
   \includegraphics[width=0.8\textwidth]{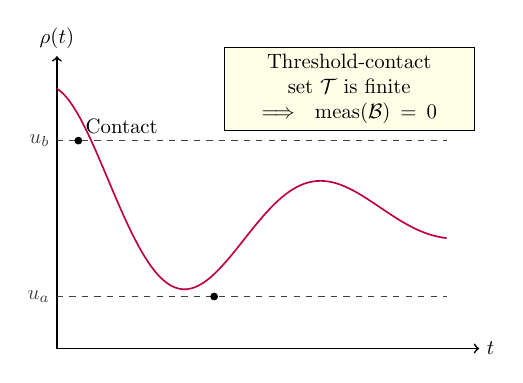}
    \caption{Illustration of the switching function $\rho(t)$ over time. In the decoupled scalar setting, the curve exhibits only a finite number of contacts with the upper and lower control thresholds, $u_b$ and $u_a$. Consequently, the threshold-contact set $\mathcal{T}$ is finite, which implies that the biactive set $\mathcal{B}$ has Lebesgue measure zero ($\operatorname{meas}(\mathcal{B})=0$). This guarantees that strict complementarity holds almost everywhere.}
    \label{fig:switching_function}
\end{figure}

\begin{corollary}[Decoupled two-agent case]
\label{cor:decoupled-two-agent-sc-app}
In the two-agent model of Section~\ref{subsec:two-agent} with $x_1=0$, the
lower-level problem separates into two independent scalar constrained LQ problems.
For agent $i$, the scalar running state weight matching the normalization of
\eqref{eq:scalar-problem-app} is
\[
  c_i^{\mathrm{sc}}=\tfrac12 x_2 c_i
\]
(equivalently, the factor $1/2$ may be absorbed into the scalar coefficient). Suppose
$b_i\neq0$, $c_i\ge0$, $\sigma>0$, $u_{a,i}<0<u_{b,i}$, $i=1,2$. Then each scalar
threshold-contact set is finite; consequently the biactive set of the two-agent
box-control system has Lebesgue measure zero, and strict complementarity holds almost
everywhere for the decoupled two-agent lower-level problem. The boundary case $x_2=0$
(so $c_i^{\mathrm{sc}}=0$) is included: by Theorem~\ref{thm:scalar-sc-app} with $c=0$
one has $p\equiv0$, $\rho\equiv0\in(u_{a,i},u_{b,i})$, no threshold contact occurs,
and $\mathcal B=\varnothing$.

In particular, the complementarity degeneracy associated with vanishing bound
multipliers is absent up to a null set. Thus the strict-complementarity component of
the MPCC--MFCQ verification used in Theorem~\ref{thm:cstationarity} is automatic in
this decoupled scalar setting; the remaining qualification reduces to the standard
smooth equality-constraint part of the state--adjoint system and the
finite-dimensional parameter constraints.
\end{corollary}

\begin{remark}[Why coupling changes the argument]
\label{rem:scalar-coupled-app}
The proof is unconditional because the simultaneous identities $\rho(t^*)=u_b$, $\dot
\rho(t^*)=0$, $\ddot\rho(t^*)=0$ overdetermine the scalar state--costate system; in
the upper-threshold case they force $a^2\sigma+2b^2c=0$, impossible when $c>0$,
$b\neq0$, $\sigma>0$ (the lower-threshold case is identical), while for $c=0$ the
adjoint vanishes identically and no threshold contact occurs. In the coupled
two-agent case $x_1>0$, the same Rolle argument yields relations among several state
and costate components rather than a single scalar algebraic contradiction, so the
system is no longer automatically overdetermined. Hence the decoupled case admits an
unconditional a.e.\ strict-complementarity certificate, whereas the coupled case
generally requires either a separate nondegeneracy assumption or an a posteriori
verification.
\end{remark}

\bibliography{reference}

\end{document}